\documentclass{amsart}
\usepackage{amssymb}
\usepackage{amsmath}
\usepackage{amsfonts}

\newtheorem{theorem}{Theorem}[section]

\newtheorem{lemma}[theorem]{Lemma}

\newtheorem{remark}[theorem]{Remark}
\numberwithin{theorem}{section}

\input{tcilatex}

\begin{document}
\title[Unbounded local completely contractive maps]{Unbounded local
completely contractive maps}
\author{Maria Joi\c{t}a}
\address{Department of Mathematics, Faculty of Applied Sciences, University
Politehnica of Bucharest, 313 Spl. Independentei, 060042, Bucharest, Romania
and, Simion Stoilow Institute of Mathematics of the Romanian Academy, P.O.
Box 1-764, 014700, Bucharest, Romania}
\email{mjoita@fmi.unibuc.ro and maria.joita@mathem.pub.ro}
\urladdr{http://sites.google.com/a/g.unibuc.ro/maria-joita/}
\subjclass[2000]{ 46L05}
\keywords{locally $C^{\ast }$-algebras, quantized domain, local completely
positive maps, local completely contractive maps}
\thanks{This work was partially supported by a grant of the Ministry of
Research, Innovation and\\
Digitization, CNCS/CCCDI--UEFISCDI, project number
PN-III-P4-ID-PCE-2020-0458, within PNCDI III..}

\begin{abstract}
We prove a local convex version of Arveson's extension theorem and of
Wittstock's extension theorem. Also we prove a Stinespring type theorem for
unbounded local completely contractive maps.
\end{abstract}

\maketitle

\section{Introduction}

A locally $C^{\ast }$-algebra is a complete Hausdorff complex topological $%
\ast $-algebra whose topology is determined by an upward filtered family of $%
C^{\ast }$-seminorms. So the notion of locally $C^{\ast }$-algebra can be
regarded as a generalization of the notion of $C^{\ast }$-algebra. The
closed $\ast $-subalgebras of $B(\mathcal{H})$, the algebra of all bounded
linear operators on a Hilbert space $\mathcal{H}$, are concrete models for $%
C^{\ast }$-algebras and the $\ast $-subalgebras of unbounded linear
operators on a Hilbert space are concrete models for locally $C^{\ast }$%
-algebras. If $\mathcal{E}$ is a quantized domain in a Hilbert space $%
\mathcal{H}$, then $C^{\ast }(\mathcal{D}_{\mathcal{E}})$, the $\ast $%
-algebra of all densely defined linear operators on $\mathcal{H}$ that
verify certain conditions, is a locally $C^{\ast }$-algebra. For every
locally $C^{\ast }$-algebra $\mathcal{A}$ there is a quantized domain $%
\mathcal{E}$ in a Hilbert space $\mathcal{H}$ and a local isometric $\ast $%
-homomorphism $\pi :\mathcal{A\rightarrow }C^{\ast }(\mathcal{D}_{\mathcal{E}%
})$ \cite[Theorem 7.2]{D}. This result can be regarded as an unbounded
analog of the Ghelfand-Naimark theorem. In the literature, the locally $%
C^{\ast }$-algebras are studied under different names like pro-$C^{\ast }$%
-algebras (D. Voiculescu, N.C. Philips), $LMC^{\ast }$-algebras (G. Lassner,
K. Schm\"{u}dgen), $b^{\ast }$-algebras (C. Apostol), multinormed $C^{\ast }$%
-algebras (A. Dosiev). The term locally $C^{\ast }$-algebra is due to A.
Inoue \cite{I}.

An element $a$ in a locally $C^{\ast }$-algebra $\mathcal{A}$ whose topology
is given by the family of $C^{\ast }$-seminorms $\{p_{\lambda }\}_{\lambda
\in \Lambda }$ is positive if there exists $b\in \mathcal{A}$ such that $%
a=bb^{\ast }$ and it is local positive if there exist $b,c\in \mathcal{A}$
such that $a=bb^{\ast }+c$ with $p_{\lambda }\left( c\right) =0$ for some $%
\lambda \in \Lambda $. Thus, the notion of completely positive map,
respectively local completely positive map appeared naturally while studying
linear maps between locally $C^{\ast }$-algebras.

Dosiev \cite{D} proved a Stinespring type theorem for unbounded local
completely contractive and local completely positive ($\mathcal{CCP}$)$\ $%
maps on unital locally $C^{\ast }$-algebras. Bhat, Ghatak and Kumar \cite%
{BGK} showed that the minimal Stinespring dilation associated to an
unbounded local $\mathcal{CCP}$-map is unique up to unitary equivalence and
introduced a partial order relation on the set of all unbounded local $%
\mathcal{CCP}$-maps in terms of their Stinespring dilations. In \cite{MJ1},
we proved a structure theorem for unbounded local $\mathcal{CCP}$-maps of
order zero and in \cite{MJ3} we proved some factorization properties for
unbounded local positive maps. In this paper, we prove a local convex
version of Arveson's extension theorem \cite{A}. As in the case of
completely bounded maps on $C^{\ast }$-algebras, we show that an unbounded
local $\mathcal{CC}$-map can be realized as the off-diagonal corners of an
unbounded local $\mathcal{CP}$-map. Using the off-diagonal technique, we
prove a local convex version of Wittstock's extension theorem and a
Stinespring type theorem for unbounded local $\mathcal{CC}$-maps. Other
local convex versions of Arverson's extension theorem and of Wittstock's
extension theorem and a structure theorem of type Stinespring can be found
in \cite{DA}.\textbf{\ }It is well-known that Arverson's extension theorem
palys a big role in the characterization of nuclear $C^{\ast }$-algebras in
terms of completly positive maps. The work developed in this paper is
essentially used in a forthcoming paper, which introduces the notion of
local nuclear multinormed $C^{\ast }$-algebras and investigates their
properties.

\section{Preliminaries}

A \textit{locally }$C^{\ast }$\textit{-algebra }is a complete Hausdorff
complex topological $\ast $-algebra $\mathcal{A}$ whose topology is
determined by an upward filtered family $\{p_{\lambda }\}_{\lambda \in
\Lambda }\ $of $C^{\ast }$-seminorms defined on $\mathcal{A}$.

A locally $C^{\ast }$-algebra $\mathcal{A}$ can be realized as a projective
limit of an inverse family of $C^{\ast }$-algebras. If $\mathcal{A}$ is a
locally $C^{\ast }$-algebra\textit{\ }with the topology determined by the
family of $C^{\ast }$-seminorms $\{p_{\lambda }\}_{\lambda \in \Lambda }$,
for each $\lambda \in \Lambda $, $\mathcal{I}_{\lambda }=\{a\in \mathcal{A}%
;p_{\lambda }\left( a\right) =0\}$ is a closed two sided $\ast $-ideal in $%
\mathcal{A}$ and $\mathcal{A}_{\lambda }=\mathcal{A}/\mathcal{I}_{\lambda }$
is a $C^{\ast }$-algebra with respect to the $C^{\ast }$-norm induced by $%
p_{\lambda }$. The canonical quotient $\ast $-morphism from $\mathcal{A}$ to 
$\mathcal{A}_{\lambda }$ is denoted by $\pi _{\lambda }^{\mathcal{A}}$. For
each $\lambda _{1},\lambda _{2}\in \Lambda $ with $\lambda _{1}\leq \lambda
_{2}$, there is a canonical surjective $\ast $-morphism $\pi _{\lambda
_{2}\lambda _{1}}^{\mathcal{A}}:$ $\mathcal{A}_{\lambda _{2}}\rightarrow 
\mathcal{A}_{\lambda _{1}},$ defined by $\pi _{\lambda _{2}\lambda _{1}}^{%
\mathcal{A}}\left( a+\mathcal{I}_{\lambda _{2}}\right) =a+\mathcal{I}%
_{\lambda _{1}}$ for $a\in \mathcal{A}$. Then, $\{\mathcal{A}_{\lambda },\pi
_{\lambda _{2}\lambda _{1}}^{\mathcal{A}}\}$\ forms an inverse system of $%
C^{\ast }$-algebras, since $\pi _{\lambda _{1}}^{\mathcal{A}}=$ $\pi
_{\lambda _{2}\lambda _{1}}^{\mathcal{A}}\circ \pi _{\lambda _{2}}^{\mathcal{%
A}}$ whenever $\lambda _{1}\leq \lambda _{2}$. The projective limit%
\begin{equation*}
\lim\limits_{\underset{\lambda }{\leftarrow }}\mathcal{A}_{\lambda
}=\{\left( a_{\lambda }\right) _{\lambda \in \Lambda }\in
\tprod\limits_{\lambda \in \Lambda }\mathcal{A}_{\lambda };\pi _{\lambda
_{2}\lambda _{1}}^{\mathcal{A}}\left( a_{\lambda _{2}}\right) =a_{\lambda
_{1}}\text{ whenever }\lambda _{1}\leq \lambda _{2},\lambda _{1},\lambda
_{2}\in \Lambda \}
\end{equation*}%
of the inverse system of $C^{\ast }$-algebras $\{\mathcal{A}_{\lambda },\pi
_{\lambda _{2}\lambda _{1}}^{\mathcal{A}}\}$ is a locally $C^{\ast }$%
-algebra that can be identified with $\mathcal{A}$ by the map $a\mapsto
\left( \pi _{\lambda }^{\mathcal{A}}\left( a\right) \right) _{\lambda \in
\Lambda }$.

An element $a\in \mathcal{A}$ is \textit{self-adjoint} if $a^{\ast }=a$ and
it is\textit{\ positive} if $a=b^{\ast }b$ for some $b\in \mathcal{A}.$

An element $a\in \mathcal{A}$ is called \textit{local self-adjoint} if $%
a=a^{\ast }+c$, where $c\in \mathcal{A}$ such that $p_{\lambda }\left(
c\right) =0$ for some $\lambda \in \Lambda $, and we call $a$ as $\lambda $%
\textit{-self-adjoint,} and \textit{local positive} if $a=b^{\ast }b+c$
where $b,c\in $ $\mathcal{A}$ such that $p_{\lambda }\left( c\right) =0\ $
for some $\lambda \in \Lambda $, we call $a$ as $\lambda $\textit{-positive}
and note $a\geq _{\lambda }0$. We write, $a=_{\lambda }0$ whenever $%
p_{\lambda }\left( a\right) =0.$

Note that $a\in \mathcal{A}$ is local self-adjoint if and only if there is $%
\lambda \in \Lambda $ such that $\pi _{\lambda }^{\mathcal{A}}\left(
a\right) $ is self adjoint in $\mathcal{A}_{\lambda }$ and $a\in \mathcal{A}$
is local positive if and only if there is $\lambda \in \Lambda $ such that $%
\pi _{\lambda }^{\mathcal{A}}\left( a\right) $ is positive in $\mathcal{A}%
_{\lambda }.$

An element $a\in \mathcal{A}$ is \textit{bounded} if $\sup \{p_{\lambda
}\left( a\right) ;\lambda \in \Lambda \}<\infty .$ Then $b\left( \mathcal{A}%
\right) =\{a\in \mathcal{A};\left\Vert a\right\Vert _{\infty }=\sup
\{p_{\lambda }\left( a\right) ;\lambda \in \Lambda \}<\infty \}\ $is a $%
C^{\ast }$-algebra with respect to the $C^{\ast }$-norm $\left\Vert \cdot
\right\Vert _{\infty }.$ Moreover, $b\left( \mathcal{A}\right) $ is dense in 
$\mathcal{A}.$ For more details on locally $C^{\ast }$-algebras we refer the
reader to \cite{F,I,Ph}.

Throughout the paper, $\mathcal{H}$ is a complex Hilbert space and $B(%
\mathcal{H})$ is the algebra of all bounded linear operators on the Hilbert
space $\mathcal{H}$.

Let $(\Upsilon ,\leq )$ be a directed poset. A \textit{quantized domain in a
Hilbert space} $\mathcal{H}$ is a triple $\{\mathcal{H};\mathcal{E};\mathcal{%
D}_{\mathcal{E}}\}$, where $\mathcal{E}=\{\mathcal{H}_{\iota };\iota \in
\Upsilon \}$ is an upward filtered family of closed subspaces such that the
union space $\mathcal{D}_{\mathcal{E}}=\tbigcup\limits_{\iota \in \Upsilon }%
\mathcal{H}_{\iota }$ is dense in $\mathcal{H\ }$\cite{D}.

Let $\mathcal{E}=\{\mathcal{H}_{\iota };\iota \in \Upsilon \}$ be a
quantized domain in a Hilbert space $\mathcal{H}\ $and $\mathcal{F}=\{%
\mathcal{K}_{\iota };\iota \in \Upsilon \}$ be a quantized domain in a
Hilbert space $\mathcal{K}\ $. Then $\mathcal{E}\oplus \mathcal{F}=\left\{ 
\mathcal{H}_{\iota }\oplus \mathcal{K}_{\iota };\iota \in \Upsilon \right\}
\ $is a quantized domain in the Hilbert space $\mathcal{H}\oplus \mathcal{K}$
with the union space $\mathcal{D}_{\mathcal{E}\oplus \mathcal{F}}=\mathcal{D}%
_{\mathcal{E}}\oplus \mathcal{D}_{\mathcal{F}}$.$\ $For each $n,\ $we\ will
use the notations: $\underset{n}{\underbrace{\mathcal{E}\oplus ...\oplus 
\mathcal{E}}}=\mathcal{E}^{n}\ $and~$\underset{n}{\underbrace{\mathcal{D}_{%
\mathcal{E}}\oplus ...\oplus \mathcal{D}_{\mathcal{E}}}}\ =\mathcal{D}_{%
\mathcal{E}^{n}}.$

Given a linear operator $V:\mathcal{D}_{\mathcal{E}}\rightarrow \mathcal{K}$%
, we write $V(\mathcal{E})\subseteq \mathcal{F}$ if $V(\mathcal{H}_{\iota
})\subseteq \mathcal{K}_{\iota }$\ for all $\iota \in \Upsilon $.

Let $\mathcal{E}=\{\mathcal{H}_{\iota };\iota \in \Upsilon \}$ be a
quantized domain in a Hilbert space $\mathcal{H}$. The quantized family $%
\mathcal{E}=\{\mathcal{H}_{\iota };\iota \in \Upsilon \}$ determines an
upward filtered family $\{P_{\iota };\iota \in \Upsilon \}$ of projections
in $B(\mathcal{H})$, where $P_{\iota }$ is a projection onto $\mathcal{H}%
_{\iota }$. Let 
\begin{equation*}
C^{\ast }(\mathcal{D}_{\mathcal{E}})=\{T\in \mathcal{L}(\mathcal{D}_{%
\mathcal{E}});TP_{\iota }=P_{\iota }TP_{\iota }\in B(\mathcal{H})\text{ and }%
P_{\iota }T\subseteq TP_{\iota }\text{ for all }\iota \in \Upsilon \}
\end{equation*}%
where $\mathcal{L}(\mathcal{D}_{\mathcal{E}})$ is the collection of all
linear operators on $\mathcal{D}_{\mathcal{E}}$. If $T\in \mathcal{L}(%
\mathcal{D}_{\mathcal{E}})$, then $T\in C^{\ast }(\mathcal{D}_{\mathcal{E}})$
if and only if $T(\mathcal{H}_{\iota })\subseteq \mathcal{H}_{\iota },T(%
\mathcal{H}_{\iota }^{\bot }\cap \mathcal{D}_{\mathcal{E}})\subseteq 
\mathcal{H}_{\iota }^{\bot }\cap \mathcal{D}_{\mathcal{E}}$ and $\left.
T\right\vert _{\mathcal{H}_{\iota }}\in B(\mathcal{H}_{\iota })$ for all $%
\iota \in \Upsilon .\ $

If $T\in $ $\mathcal{L}(\mathcal{D}_{\mathcal{E}})$, then $T$ is a densely
defined linear operator on $\mathcal{H}$. The adjoint of $T$ is a linear map 
$T^{\bigstar }:$ dom$(T^{\bigstar })$ $\subseteq $ $\mathcal{H}\rightarrow 
\mathcal{H}$, where%
\begin{equation*}
\text{dom}(T^{\bigstar })=\{\xi \in \mathcal{K};\eta \rightarrow
\left\langle T\eta ,\xi \right\rangle \ \text{is continuous for every }\eta
\in \text{dom}(T)\}
\end{equation*}%
and satisfying $\left\langle T\eta ,\xi \right\rangle =\left\langle \eta
,T^{\bigstar }\xi \right\rangle $ for all $\xi \in $dom$(T^{\bigstar })$ and 
$\eta \in $dom$(T).$

If $T\in C^{\ast }(\mathcal{D}_{\mathcal{E}})$, then $\mathcal{D}_{\mathcal{E%
}}$ $\subseteq $ dom$(T^{\bigstar })$, and $\left. T^{\bigstar }\right\vert
_{\mathcal{D}_{\mathcal{E}}}\in C^{\ast }(\mathcal{D}_{\mathcal{E}})$. Let $%
T^{\ast }=\left. T^{\bigstar }\right\vert _{\mathcal{D}_{\mathcal{E}}}$. It
is easy to check that $C^{\ast }(\mathcal{D}_{\mathcal{E}})$ is a unital $%
\ast $-algebra. For each $\iota \in \Upsilon ,$ the map $\left\Vert \cdot
\right\Vert _{\iota }:C^{\ast }(\mathcal{D}_{\mathcal{E}})\rightarrow
\lbrack 0,\infty )$, 
\begin{equation*}
\left\Vert T\right\Vert _{\iota }=\left\Vert \left. T\right\vert _{\mathcal{H%
}_{\iota }}\right\Vert =\sup \{\left\Vert T\left( \xi \right) \right\Vert
;\xi \in \mathcal{H}_{\iota },\left\Vert \xi \right\Vert \leq 1\}
\end{equation*}%
is a $C^{\ast }$-seminorm on $C^{\ast }(\mathcal{D}_{\mathcal{E}})$.
Therefore, $C^{\ast }(\mathcal{D}_{\mathcal{E}})$ is a locally $C^{\ast }$%
-algebra with respect to the family of $C^{\ast }$-seminorms $\{\left\Vert
\cdot \right\Vert _{\iota }\}_{\iota \in \Upsilon }$, and  $b(C^{\ast }(%
\mathcal{D}_{\mathcal{E}}))$ can be identified with the $C^{\ast }$-algebra $%
\{T\in B\left( \mathcal{H}\right) ;P_{\lambda }T=TP_{\lambda }\ $for all $%
\lambda \in \Lambda \}$\ via\ the map $T\mapsto \widetilde{T}$, where $%
\widetilde{T}$ is extension of $T$ to $\mathcal{H}.$ If $\mathcal{E}=\{%
\mathcal{H}\}$, then $C^{\ast }(\mathcal{D}_{\mathcal{E}})=B(\mathcal{H})$.

For each $n\in \mathbb{N},$ $M_{n}(\mathcal{A})$ denotes the collection of
all matrices of order $n$ with elements in $\mathcal{A}$. Note that $M_{n}(%
\mathcal{A})$ is a locally $C^{\ast }$-algebra, the associated family of $%
C^{\ast }$-seminorms being denoted by $\{p_{\lambda }^{n}\}_{\lambda \in
\Lambda },$ and $M_{n}(C^{\ast }(\mathcal{D}_{\mathcal{E}}))$ can be
identified with $C^{\ast }(\mathcal{D}_{\mathcal{E}^{n}}).$

For each $n\in \mathbb{N}$, the $n$-amplification of a linear map $\varphi :%
\mathcal{A}\rightarrow C^{\ast }(\mathcal{D}_{\mathcal{E}})$ is the map $%
\varphi ^{\left( n\right) }:M_{n}(\mathcal{A})$ $\rightarrow $ $C^{\ast }(%
\mathcal{D}_{\mathcal{E}^{n}})$ defined by 
\begin{equation*}
\varphi ^{\left( n\right) }\left( \left[ a_{ij}\right] _{i,j=1}^{n}\right) =%
\left[ \varphi \left( a_{ij}\right) \right] _{i,j=1}^{n}
\end{equation*}%
for all $\left[ a_{ij}\right] _{i,j=1}^{n}\in M_{n}(\mathcal{A})$ .

A linear map $\varphi :\mathcal{A}\rightarrow C^{\ast }(\mathcal{D}_{%
\mathcal{E}})$ is called :

\begin{enumerate}
\item \textit{local contractive} if for each $\iota \in \Upsilon $, there
exists $\lambda \in \Lambda $ such that%
\begin{equation*}
\left\Vert \left. \varphi \left( a\right) \right\vert _{\mathcal{H}_{\iota
}}\right\Vert \leq p_{\lambda }\left( a\right) \text{ for all }a\in \mathcal{%
A};
\end{equation*}

\item \textit{local completely contractive (local }$\mathcal{CC}$\textit{) }%
if for each $\iota \in \Upsilon $, there exists $\lambda \in \Lambda $ such
that 
\begin{equation*}
\left\Vert \left. \varphi ^{\left( n\right) }\left( \left[ a_{ij}\right]
_{i,j=1}^{n}\right) \right\vert _{\mathcal{H}_{\iota }^{n}}\right\Vert \leq
p_{\lambda }^{n}\left( \left[ a_{ij}\right] _{i,j=1}^{n}\right) \text{ }
\end{equation*}%
for all $\left[ a_{ij}\right] _{i,j=1}^{n}\in M_{n}(\mathcal{A})\ $and for
all $n\in \mathbb{N};$

\item \textit{positive} if $\varphi \left( a\right) \ $is positive in $%
C^{\ast }(\mathcal{D}_{\mathcal{E}})\ $whenever $a\ $is positive in $%
\mathcal{A};$

\item \textit{completely positive }if $\varphi ^{\left( n\right) }\left( %
\left[ a_{ij}\right] _{i,j=1}^{n}\right) $ is positive in $C^{\ast }(%
\mathcal{D}_{\mathcal{E}^{n}})$ whenever $\left[ a_{ij}\right] _{i,j=1}^{n}$
is positive in $M_{n}(\mathcal{A})$ for all $n\in \mathbb{N};$

\item \textit{local positive} if for each $\iota \in \Upsilon $, there
exists $\lambda \in \Lambda $ such that $\left. \varphi \left( a\right)
\right\vert _{\mathcal{H}_{\iota }}\ $is positive in $B\left( \mathcal{H}%
_{\iota }\right) \ $whenever $a\geq _{\lambda }0$ and $\left. \varphi \left(
a\right) \right\vert _{\mathcal{H}_{\iota }}=0\ $whenever $a=_{\lambda }0;$

\item \textit{local completely positive (local }$\mathcal{CP}$\textit{) }if
for each $\iota \in \Upsilon $, there exists $\lambda \in \Lambda $ such
that $\left. \varphi ^{\left( n\right) }\left( \left[ a_{ij}\right]
_{i,j=1}^{n}\right) \right\vert _{\mathcal{H}_{\iota }^{n}}\ $is positive in 
$B\left( \mathcal{H}_{\iota }^{n}\right) $ whenever $\left[ a_{ij}\right]
_{i,j=1}^{n}\geq _{\lambda }0$ and $\left. \varphi ^{\left( n\right) }\left( %
\left[ a_{ij}\right] _{i,j=1}^{n}\right) \right\vert _{\mathcal{H}_{\iota
}^{n}}=0\ \ $whenever $\left[ a_{ij}\right] _{i,j=1}^{n}=_{\lambda }0$,$\ $%
for all $n\in \mathbb{N}.$
\end{enumerate}

\section{Arveson's extension theorem}

\textit{\ }

Let $\{\mathcal{H},\mathcal{E=\{H}_{n}\mathcal{\}}_{n},\mathcal{D}_{\mathcal{%
E}}\}$ be a quantized Fr\'{e}chet domain in $\mathcal{H},P_{n}$ be the
projection from $\mathcal{H}$ onto $\mathcal{H}_{n},n\geq 1$ and $%
S_{n}=\left( \text{id}_{\mathcal{H}}-P_{n-1}\right) P_{n}$ be the projection
onto $\mathcal{H}_{n-1}^{\perp }\cap \mathcal{H}_{n},$ $n\geq 2$ and $%
S_{1}=P_{1}$. By \cite[Proposition 4.2]{D}, if $T\in C^{\ast }(\mathcal{D}_{%
\mathcal{E}}),$ then $T$ $\ $has a diagonal representation, $%
T=\tsum\limits_{n=1}^{\infty }S_{n}TS_{n}.$

Let $\mathcal{A}$ be a unital locally $C^{\ast }$-algebra with the topology
defined by the family of $C^{\ast }$-seminorms $\{p_{\lambda }\}_{\lambda
\in \Lambda }$. A self-adjoint subspace $\mathcal{S}$ containing $1_{%
\mathcal{A}}$ is called a\textit{\ local operator system. }

\begin{remark}
\label{help} Let $\mathcal{A}$ be a unital locally $C^{\ast }$-algebra, $%
\mathcal{S\subseteq }$ $\mathcal{A}$ be a local operator system and $\varphi
:\mathcal{S\rightarrow }B\mathcal{(H)}$ be a local $\mathcal{CP}$-map. Then
there exist $\lambda \in \Lambda $ and a $\mathcal{CP}$-map $\varphi
_{\lambda }:\mathcal{S}_{\lambda }\rightarrow B\mathcal{(H)}$ such that $%
\varphi =\varphi _{\lambda }\circ \pi _{\lambda }^{\mathcal{A}}$, where $%
\mathcal{S}_{\lambda }=\pi _{\lambda }^{\mathcal{A}}\left( \mathcal{S}%
\right) .$

Indeed, since $\varphi :\mathcal{S\rightarrow }B\mathcal{(H)}$ is a local $%
\mathcal{CP}$-map there exists $\lambda \in \Lambda $ such that $\varphi
\left( a\right) $ is positive in $B\mathcal{(H)}$ whenever $\pi _{\lambda }^{%
\mathcal{A}}\left( a\right) $ is positive in $\mathcal{S}_{\lambda }$ and $%
\varphi \left( a\right) =0$ if $\pi _{\lambda }^{\mathcal{A}}\left( a\right)
=0$. Therefore, there exists a linear map $\varphi _{\lambda }^{+}:\left( 
\mathcal{S}_{\lambda }\right) _{+}\rightarrow B\mathcal{(H)}$ such that $%
\varphi _{\lambda }^{+}\left( \pi _{\lambda }^{\mathcal{A}}\left( a\right)
\right) =\varphi \left( a\right) .$

If $a\in \mathcal{S}$ is a $\lambda $-self-adjoint element, then $p_{\lambda
}\left( a\right) 1_{\mathcal{A}}\pm a\geq _{\lambda }0$. Clearly, $%
p_{\lambda }\left( a\right) 1_{\mathcal{A}}\pm a\in \mathcal{S}$ and $a=%
\frac{1}{2}\left( p_{\lambda }\left( a\right) 1_{\mathcal{A}}+a\right) -%
\frac{1}{2}\left( p_{\lambda }\left( a\right) 1_{\mathcal{A}}-a\right) $.
Therefore, any $\lambda $-self-adjoint element from $\mathcal{S}$ is the
difference of two $\lambda $-positive elements from $\mathcal{S}$. On the
other hand, since $\mathcal{S}$ is a self-adjoint subspace of $\mathcal{A},$
any element from $\mathcal{S}$ is the sum of two $\lambda $-self-adjoint
elements from $\mathcal{S}.$ Therefore, the map $\varphi _{\lambda }^{+}$
extends by linearity to a positive linear map $\varphi _{\lambda }:\mathcal{S%
}_{\lambda }\rightarrow B\mathcal{(H)}$ such that $\varphi _{\lambda }\left(
\pi _{\lambda }^{\mathcal{A}}\left( a\right) \right) =\varphi \left(
a\right) $. Clearly, $\varphi _{\lambda }$ is completely positive.
\end{remark}

The next theorem is a local convex version of the well-known Arveson's
extension theorem \cite[Theorem 1.2.3]{A}.

\begin{theorem}
\label{A} Let $\mathcal{A}$ be a unital locally $C^{\ast }$-algebra, $%
\mathcal{S\subseteq }$ $\mathcal{A}$ be a local operator system, $\{\mathcal{%
H},\mathcal{E=\{H}_{n}\mathcal{\}}_{n},\mathcal{D}_{\mathcal{E}}\}\ $be a
quantized Fr\'{e}chet domain in a Hilbert space $\mathcal{H}$ and $\varphi :%
\mathcal{S\rightarrow }C^{\ast }\left( \mathcal{D}_{\mathcal{E}}\right) $ be
a local $\mathcal{CP}$-map. Then there exists a local $\mathcal{CP}$-map $%
\widetilde{\varphi }:$ $\mathcal{A\rightarrow }C^{\ast }\left( \mathcal{D}_{%
\mathcal{E}}\right) $ extending $\varphi .$
\end{theorem}

\begin{proof}
We divided the proof into two steps.

First, we suppose that $\mathcal{E=H}$. Then $C^{\ast }\left( \mathcal{D}_{%
\mathcal{E}}\right) =B\mathcal{(H)\ }$and since $\varphi :$ $\mathcal{%
S\rightarrow }B\mathcal{(H)}$ is a local $\mathcal{CP}$-map, by Remark \ref%
{help}, there exist $\lambda \in \Lambda \ $and a $\mathcal{CP}$-map $%
\varphi _{\lambda }:$ $\mathcal{S}_{\lambda }\mathcal{\rightarrow }B\mathcal{%
(H)}$ such that $\varphi \left( a\right) =\varphi _{\lambda }\left( \pi
_{\lambda }^{\mathcal{A}}\left( a\right) \right) $ for all $a\in \mathcal{A}$%
. Since $\mathcal{S}_{\lambda }\subseteq \mathcal{A}_{\mathcal{\lambda }}$
is an operator system and $\varphi _{\lambda }:$ $\mathcal{S}_{\lambda }%
\mathcal{\rightarrow }B\mathcal{(H)}$ is a $\mathcal{CP}$-map, from
Arveson's extension theorem \cite[Theorem 1.2.3]{A}, it follows that there
exists a $\mathcal{CP}$-map $\widetilde{\varphi _{\lambda }}:\mathcal{A}%
_{\lambda }\mathcal{\rightarrow }B\mathcal{(H)}$ extending $\varphi
_{\lambda }$. Let $\widetilde{\varphi }=\widetilde{\varphi _{\lambda }}\circ
\pi _{\lambda }^{\mathcal{A}}$. Clearly, $\widetilde{\varphi }$ is a local $%
\mathcal{CP}$-map from $\mathcal{A}$ to $B\mathcal{(H)}$ and $\left. 
\widetilde{\varphi }\right\vert _{\mathcal{S}}=\varphi .$

Now we prove the general case.

For each $n$, consider the map $\varphi _{n}:$ $\mathcal{S\rightarrow }B%
\mathcal{(H)}$ defined by $\varphi _{n}\left( a\right) =P_{n}\varphi \left(
a\right) P_{n}$. Since, for each $k,$ $\varphi _{n}^{\left( k\right) }\left( %
\left[ a_{ij}\right] _{i,j=1}^{k}\right) =P_{n}^{\oplus k}\varphi ^{\left(
k\right) }\left( \left[ a_{ij}\right] _{i,j=1}^{k}\right) P_{n}^{\oplus k}$,
and $\varphi ^{\left( k\right) }$ is local positive, it follows that $%
\varphi _{n}^{\left( k\right) }$ is local positive. Therefore, $\varphi _{n}$
is a local $\mathcal{CP}$-map, and by the first part of the proof, there
exists a local $\mathcal{CP}$-map $\widetilde{\varphi _{n}}:\mathcal{%
A\rightarrow }$ $B\mathcal{(H)}$ extending $\varphi _{n}$. Consider the map $%
\widetilde{\varphi }:\mathcal{A\rightarrow }C^{\ast }\left( \mathcal{D}_{%
\mathcal{E}}\right) $ defined by 
\begin{equation*}
\widetilde{\varphi }\left( a\right) =\tsum\limits_{n=1}^{\infty }S_{n}%
\widetilde{\varphi _{n}}\left( a\right) S_{n}.
\end{equation*}%
For each $n,$%
\begin{equation*}
\widetilde{\varphi }\left( a\right) P_{n}=\tsum\limits_{k=1}^{\infty }S_{k}%
\widetilde{\varphi _{k}}\left( a\right)
S_{k}P_{n}=\tsum\limits_{k=1}^{n}S_{k}\widetilde{\varphi _{k}}\left(
a\right) S_{k}
\end{equation*}%
whence, it follows that $\widetilde{\varphi }\left( a\right) \left( \mathcal{%
H}_{n}\right) \subseteq $ $\mathcal{H}_{n}$ and $\left. \widetilde{\varphi }%
\left( a\right) \right\vert _{\mathcal{H}_{n}}\in B\left( \mathcal{H}%
_{n}\right) $. In the same way, we deduce that $\widetilde{\varphi }\left(
a^{\ast }\right) \left( \mathcal{H}_{n}\right) \subseteq $ $\mathcal{H}_{n},$
$\left. \widetilde{\varphi }\left( a^{\ast }\right) \right\vert _{\mathcal{H}%
_{n}}\in B\left( \mathcal{H}_{n}\right) $ and $\left. \widetilde{\varphi }%
\left( a\right) \right\vert _{\mathcal{H}_{n}}^{\ast }=\left. \widetilde{%
\varphi }\left( a^{\ast }\right) \right\vert _{\mathcal{H}_{n}}$. Therefore, 
$\widetilde{\varphi }$ is well defined.

Since, for each $k\in \{1,2,...,n\},$ $\widetilde{\varphi _{k}}$ $\ $is a
local $\mathcal{CP}$-map, there exists $\lambda _{n}\in \Lambda $ such that $%
\ \left. \widetilde{\varphi _{k}}^{\left( m\right) }\left( \left[ a_{ij}%
\right] _{i,j=1}^{m}\right) \right\vert _{\mathcal{H}_{n}^{m}}$ is positive
in $B\left( \mathcal{H}_{n}^{m}\right) \ $whenever $\left[ a_{ij}\right]
_{i,j=1}^{m}$ $\geq _{\lambda _{n}}0$, and $\left. \widetilde{\varphi _{k}}%
^{\left( m\right) }\left( \left[ a_{ij}\right] _{i,j=1}^{m}\right)
\right\vert _{\mathcal{H}_{n}^{m}}=0$ whenever $\left[ a_{ij}\right]
_{i,j=1}^{m}$ $=_{\lambda _{n}}0,$ for all $k\in \{1,2,...,n\}$ and for all $%
m.$ Therefore, there exists $\lambda _{n}\in \Lambda $ such that 
\begin{equation*}
\left. \widetilde{\varphi }^{\left( m\right) }\left( \left[ a_{ij}\right]
_{i,j=1}^{m}\right) \right\vert _{\mathcal{H}_{n}^{m}}=\tsum%
\limits_{k=1}^{n}\left. S_{k}^{\oplus m}\widetilde{\varphi _{k}}^{\left(
m\right) }\left( \left[ a_{ij}\right] _{i,j=1}^{m}\right) S_{k}^{\oplus
m}\right\vert _{\mathcal{H}_{n}^{m}}
\end{equation*}%
is positive in $B\left( \mathcal{H}_{n}^{m}\right) $ whenever $\left[ a_{ij}%
\right] _{i,j=1}^{m}$ $\geq _{\lambda _{n}}0$, and$\left. \widetilde{\varphi 
}^{\left( m\right) }\left( \left[ a_{ij}\right] _{i,j=1}^{m}\right)
\right\vert _{\mathcal{H}_{n}^{m}}=0$ if $\left[ a_{ij}\right] _{i,j=1}^{m}$ 
$=_{\lambda _{n}}0$. Thus, we showed that $\widetilde{\varphi }$ is a local $%
\mathcal{CP}$-map.

It remains to show that $\widetilde{\varphi }$ extends $\varphi $. Let $a\in 
\mathcal{S}$. From 
\begin{eqnarray*}
\widetilde{\varphi }\left( a\right) P_{n} &=&\tsum\limits_{k=1}^{n}S_{k}%
\widetilde{\varphi _{k}}\left( a\right)
S_{k}=\tsum\limits_{k=1}^{n}S_{k}\varphi _{k}\left( a\right)
S_{k}=\tsum\limits_{k=1}^{n}S_{k}P_{k}\varphi \left( a\right) P_{k}S_{k} \\
&=&\tsum\limits_{k=1}^{n}S_{k}\varphi \left( a\right)
S_{k}=\tsum\limits_{k=1}^{n}\varphi \left( a\right) S_{k}=\varphi \left(
a\right) P_{n}
\end{eqnarray*}%
for all $n$, we deduce that $\widetilde{\varphi }$ extends $\varphi $, and
the theorem is proved.
\end{proof}

Arunkumar proved a particular case of Arveson's extension theorem \cite[%
Theorem 3.6]{Ar}.

\section{Wittstock's extension theorem}

Let $\mathcal{A}$ be a unital locally $C^{\ast }$-algebra with the topology
defined by the family of $C^{\ast }$-seminorms $\{p_{\lambda }\}_{\lambda
\in \Lambda }$. A subspace $\mathcal{M}$ of $\mathcal{A}$ is called \textit{%
a local operator space.}

Let $\mathcal{M\subseteq }$ $\mathcal{A}$ be a local operator space. Then%
\begin{equation*}
S_{\mathcal{M}}=\left\{ \left[ 
\begin{array}{cc}
\alpha 1_{\mathcal{A}} & a \\ 
b^{\ast } & \beta 1_{\mathcal{A}}%
\end{array}%
\right] ;\alpha ,\beta \in \mathbb{C},a,b\in \mathcal{M}\right\} \subseteq
M_{2}\left( \mathcal{A}\right)
\end{equation*}
is a local operator system.

It is clear that for each $n$, $M_{n}\left( M_{2}\left( \mathcal{A}\right)
\right) $ can be identified with $M_{2}\left( M_{n}\left( \mathcal{A}\right)
\right) .$

The following lemma is central to the next results.

\begin{lemma}
\label{1} Let $\mathcal{A}$ be a unital locally $C^{\ast }$-algebra, $%
\mathcal{M\subseteq }$ $\mathcal{A}$ be a local operator space and $\varphi :%
\mathcal{M\rightarrow }C^{\ast }\left( \mathcal{D}_{\mathcal{E}}\right) \ $%
be a local $\mathcal{CC}$-map. Then the map $\Phi :S_{\mathcal{M}%
}\rightarrow C^{\ast }\left( \mathcal{D}_{\mathcal{E\oplus E}}\right) ,$%
\begin{equation*}
\Phi \left( \left[ 
\begin{array}{cc}
\alpha 1_{\mathcal{A}} & a \\ 
b^{\ast } & \beta 1_{\mathcal{A}}%
\end{array}%
\right] \right) =\left[ 
\begin{array}{cc}
\alpha 1_{C^{\ast }\left( \mathcal{D}_{\mathcal{E}}\right) } & \varphi
\left( a\right) \\ 
\varphi \left( b\right) ^{\ast } & \beta 1_{C^{\ast }\left( \mathcal{D}_{%
\mathcal{E}}\right) }%
\end{array}%
\right]
\end{equation*}%
is a local $\mathcal{CP}$-map.
\end{lemma}

\begin{proof}
Since $\varphi $ is a local $\mathcal{CC}$-map, for each $\iota \in \Upsilon
,$ there exists $\lambda \in \Lambda $ such that $\left\Vert \left. \varphi
^{\left( n\right) }\left( A\right) \right\vert _{\mathcal{H}_{\tau
}^{n}}\right\Vert \leq p_{\lambda }^{n}\left( A\right) $ for all $A\in
M_{n}\left( \mathcal{M}\right) $ and for all $n.$ Since $M_{n}\left( S_{%
\mathcal{M}}\right) $ is a subspace of $M_{n}\left( M_{2}\left( \mathcal{A}%
\right) \right) ,$ it can be identified with a subspace in $M_{2}\left(
M_{n}\left( \mathcal{A}\right) \right) $.$\ $Thus, an element $X\ $in $%
M_{n}\left( S_{\mathcal{M}}\right) $ is of the forma $\left[ 
\begin{array}{cc}
\left[ \alpha _{ij}1_{\mathcal{A}}\right] _{i,j=1}^{n} & \left[ a_{ij}\right]
_{i,j=1}^{n} \\ 
\left( \left[ b_{ij}\right] _{i,j=1}^{n}\right) ^{\ast } & \left[ \beta
_{ij}1_{\mathcal{A}}\right] _{i,j=1}^{n}%
\end{array}%
\right] .$ To show that $\Phi $ is a local $\mathcal{CP}$-map, we must show
that for each $\iota \in \Upsilon ,$ there exists $\lambda \in \Lambda $
such that $\left. \Phi ^{\left( n\right) }\left( X\right) \right\vert _{%
\mathcal{H}_{\iota }^{n}\oplus \mathcal{H}_{\iota }^{n}}\geq 0\ $whenever $%
X\geq _{\lambda }0$, and $\left. \Phi ^{\left( n\right) }\left( X\right)
\right\vert _{\mathcal{H}_{\iota }^{n}\oplus \mathcal{H}_{\iota }^{n}}=0$
whenever $X=_{\lambda }0,$ for all $n.$

Let$\ \iota \in \Upsilon $ and $n\in \mathbb{N}^{\ast }$. If $X=\left[ 
\begin{array}{cc}
\left[ \alpha _{ij}1_{\mathcal{A}}\right] _{i,j=1}^{n} & \left[ a_{ij}\right]
_{i,j=1}^{n} \\ 
\left( \left[ b_{ij}\right] _{i,j=1}^{n}\right) ^{\ast } & \left[ \beta
_{ij}1_{\mathcal{A}}\right] _{i,j=1}^{n}%
\end{array}%
\right] \geq _{\lambda }0$, then%
\begin{equation*}
\left[ 
\begin{array}{cc}
\left[ \alpha _{ij}1_{\mathcal{A}_{\lambda }}\right] _{i,j=1}^{n} & \left[
\pi _{\lambda }^{\mathcal{A}}\left( a_{ij}\right) \right] _{i,j=1}^{n} \\ 
\left( \left[ \pi _{\lambda }^{\mathcal{A}}\left( b_{ij}\right) \right]
_{i,j=1}^{n}\right) ^{\ast } & \left[ \beta _{ij}1_{\mathcal{A}_{\lambda }}%
\right] _{i,j=1}^{n}%
\end{array}%
\right] \geq 0
\end{equation*}%
in $M_{2}\left( M_{n}\left( \mathcal{A}_{\lambda }\right) \right) $.
Consequently, $\left[ \pi _{\lambda }^{\mathcal{A}}\left( a_{ij}\right) %
\right] _{i,j=1}^{n}=\left[ \pi _{\lambda }^{\mathcal{A}}\left(
b_{ij}\right) \right] _{i,j=1}^{n}=A$ and the matrices $\left[ \alpha _{ij}%
\right] _{i,j=1}^{n}$ and $\left[ \beta _{ij}\right] _{i,j=1}^{n}$ are
positive. As in the proof of \cite[Lemma 8.1]{P}, for all $\varepsilon >0,$
the matrices $\alpha _{\varepsilon }=\left[ \alpha _{ij}1_{\mathcal{A}%
_{\lambda }}\right] _{i,j=1}^{n}+\varepsilon 1_{M_{n}\left( \mathcal{A}%
_{\lambda }\right) }$ and $\beta _{\varepsilon }=\left[ \beta _{ij}1_{%
\mathcal{A}_{\lambda }}\right] _{i,j=1}^{n}+\varepsilon 1_{M_{n}\left( 
\mathcal{A}_{\lambda }\right) }\ $are positive and invertible, and $%
\left\Vert \alpha _{\varepsilon }^{-\frac{1}{2}}A\beta _{\varepsilon }^{-%
\frac{1}{2}}\right\Vert _{M_{n}\left( \mathcal{A}_{\lambda }\right) }\leq 1$.

Let $\widetilde{\alpha _{\varepsilon }}=\left[ \alpha _{ij}\text{id}_{%
\mathcal{H}_{\iota }}\right] _{i,j=1}^{n}+\varepsilon $id$_{\mathcal{H}%
_{\iota }^{n}},$ $\widetilde{\beta _{\varepsilon }}=\left[ \beta _{ij}\text{%
id}_{\mathcal{H}_{\iota }}\right] _{i,j=1}^{n}+\varepsilon $id$_{\mathcal{H}%
_{\iota }^{n}}\ $and 
\begin{equation*}
H_{\varepsilon }=\left[ 
\begin{array}{cc}
\text{id}_{\mathcal{H}_{\iota }^{n}} & \widetilde{\alpha _{\varepsilon }}^{%
\frac{-1}{2}}\left. \varphi ^{\left( n\right) }\left( \left[ a_{ij}\right]
_{i,j=1}^{n}\right) \right\vert _{\mathcal{H}_{\iota }^{n}}\widetilde{\beta
_{\varepsilon }}^{\frac{-1}{2}} \\ 
\widetilde{\beta _{\varepsilon }}^{\frac{-1}{2}}\left. \left( \varphi
^{\left( n\right) }\left( \left[ a_{ij}\right] _{i,j=1}^{n}\right) \right)
^{\ast }\right\vert _{\mathcal{H}_{\iota }^{n}}\widetilde{\alpha
_{\varepsilon }}^{\frac{-1}{2}} & \text{id}_{\mathcal{H}_{\iota }^{n}}%
\end{array}%
\right] .
\end{equation*}

Since 
\begin{eqnarray*}
&&\widetilde{\alpha _{\varepsilon }}^{-\frac{1}{2}}\left. \varphi ^{\left(
n\right) }\left( \left[ a_{ij}\right] _{i,j=1}^{n}\right) \right\vert _{%
\mathcal{H}_{\iota }^{n}}\widetilde{\beta _{\varepsilon }}^{-\frac{1}{2}} \\
&=&\left. \varphi ^{\left( n\right) }\left( \left( \left[ \alpha _{ij}1_{%
\mathcal{A}}\right] _{i,j=1}^{n}+\varepsilon 1_{M_{n}\left( \mathcal{A}%
\right) }\right) ^{-\frac{1}{2}}\left[ a_{ij}\right] _{i,j=1}^{n}\left( %
\left[ \beta _{ij}1_{\mathcal{A}}\right] _{i,j=1}^{n}+\varepsilon
1_{M_{n}\left( \mathcal{A}\right) }\right) ^{-\frac{1}{2}}\right)
\right\vert _{\mathcal{H}_{\iota }^{n}}
\end{eqnarray*}%
and%
\begin{equation*}
\left\Vert \left. \varphi ^{\left( n\right) }\left( \left( \left[ \alpha
_{ij}1_{\mathcal{A}}\right] _{i,j=1}^{n}+\varepsilon 1_{M_{n}\left( \mathcal{%
A}\right) }\right) ^{-\frac{1}{2}}\left[ a_{ij}\right] _{i,j=1}^{n}\left( %
\left[ \beta _{ij}1_{\mathcal{A}}\right] _{i,j=1}^{n}+\varepsilon
1_{M_{n}\left( \mathcal{A}\right) }\right) ^{-\frac{1}{2}}\right)
\right\vert _{\mathcal{H}_{\iota }^{n}}\right\Vert
\end{equation*}%
\begin{eqnarray*}
&\leq &p_{\lambda }^{n}\left( \left( \left[ \alpha _{ij}1_{\mathcal{A}}%
\right] _{i,j=1}^{n}+\varepsilon 1_{M_{n}\left( \mathcal{A}\right) }\right)
^{-\frac{1}{2}}\left[ a_{ij}\right] _{i,j=1}^{n}\left( \left[ \beta _{ij}1_{%
\mathcal{A}}\right] _{i,j=1}^{n}+\varepsilon 1_{M_{n}\left( \mathcal{A}%
\right) }\right) ^{-\frac{1}{2}}\right) \\
&=&\left\Vert \alpha _{\varepsilon }^{-\frac{1}{2}}A\beta _{\varepsilon }^{-%
\frac{1}{2}}\right\Vert _{M_{n}\left( \mathcal{A}_{\lambda }\right) }\leq 1
\end{eqnarray*}%
it follows that%
\begin{equation*}
H_{\varepsilon }=\left[ 
\begin{array}{cc}
\text{id}_{\mathcal{H}_{\iota }^{n}} & \widetilde{\alpha _{\varepsilon }}^{%
\frac{-1}{2}}\left. \varphi ^{\left( n\right) }\left( \left[ a_{ij}\right]
_{i,j=1}^{n}\right) \right\vert _{\mathcal{H}_{\iota }^{n}}\widetilde{\beta
_{\varepsilon }}^{\frac{-1}{2}} \\ 
\widetilde{\beta _{\varepsilon }}^{\frac{-1}{2}}\left. \left( \varphi
^{\left( n\right) }\left( \left[ a_{ij}\right] _{i,j=1}^{n}\right) \right)
^{\ast }\right\vert _{\mathcal{H}_{\iota }^{n}}\widetilde{\alpha
_{\varepsilon }}^{\frac{-1}{2}} & \text{id}_{\mathcal{H}_{\iota }^{n}}%
\end{array}%
\right] \geq 0
\end{equation*}%
in $B\left( \mathcal{H}_{\iota }^{n}\oplus \mathcal{H}_{\iota }^{n}\right)
.\ $From 
\begin{eqnarray*}
&&\left. \Phi ^{\left( n\right) }\left( X\right) \right\vert _{\mathcal{H}%
_{\iota }^{n}\oplus \mathcal{H}_{\iota }^{n}}+\varepsilon \left[ 
\begin{array}{cc}
\text{id}_{\mathcal{H}_{\iota }^{n}} & 0 \\ 
0 & \text{id}_{\mathcal{H}_{\iota }^{n}}%
\end{array}%
\right] \\
&=&\left[ 
\begin{array}{cc}
\left[ \alpha _{ij}\text{id}_{\mathcal{H}_{\iota }}\right] _{i,j=1}^{n} & 
\left. \varphi ^{\left( n\right) }\left( \left[ a_{ij}\right]
_{i,j=1}^{n}\right) \right\vert _{\mathcal{H}_{\iota }^{n}} \\ 
\left. \left( \varphi ^{\left( n\right) }\left( \left[ a_{ij}\right]
_{i,j=1}^{n}\right) \right) ^{\ast }\right\vert _{\mathcal{H}_{\iota }^{n}}
& \left[ \beta _{ij}1_{\mathcal{H}_{\iota }}\right] _{i,j=1}^{n}%
\end{array}%
\right] +\varepsilon \left[ 
\begin{array}{cc}
\text{id}_{\mathcal{H}_{\iota }^{n}} & 0 \\ 
0 & \text{id}_{\mathcal{H}_{\iota }^{n}}%
\end{array}%
\right] \\
&=&\left[ 
\begin{array}{cc}
\widetilde{\alpha _{\varepsilon }}^{\frac{1}{2}} & 0 \\ 
0 & \widetilde{\beta _{\varepsilon }}^{\frac{1}{2}}%
\end{array}%
\right] H_{\varepsilon }\left[ 
\begin{array}{cc}
\widetilde{\alpha _{\varepsilon }}^{\frac{1}{2}} & 0 \\ 
0 & \widetilde{\beta _{\varepsilon }}^{\frac{1}{2}}%
\end{array}%
\right] \geq 0
\end{eqnarray*}%
in $B\left( \mathcal{H}_{\iota }^{n}\oplus \mathcal{H}_{\iota }^{n}\right) $
for all $\varepsilon >0,$ we deduce that $\left. \Phi ^{\left( n\right)
}\left( X\right) \right\vert _{\mathcal{H}_{\iota }^{n}\oplus \mathcal{H}%
_{\iota }^{n}}\geq 0$ in $B\left( \mathcal{H}_{\iota }^{n}\oplus \mathcal{H}%
_{\iota }^{n}\right) .$

If $X=\left[ 
\begin{array}{cc}
\left[ \alpha _{ij}1_{\mathcal{A}}\right] _{i,j=1}^{n} & \left[ a_{ij}\right]
_{i,j=1}^{n} \\ 
\left( \left[ b_{ij}\right] _{i,j=1}^{n}\right) ^{\ast } & \left[ \beta
_{ij}1_{\mathcal{A}}\right] _{i,j=1}^{n}%
\end{array}%
\right] =_{\lambda }0$, then%
\begin{equation*}
\left[ 
\begin{array}{cc}
\left[ \alpha _{ij}1_{\mathcal{A}_{\lambda }}\right] _{i,j=1}^{n} & \left[
\pi _{\lambda }^{\mathcal{A}}\left( a_{ij}\right) \right] _{i,j=1}^{n} \\ 
\left( \left[ \pi _{\lambda }^{\mathcal{A}}\left( b_{ij}\right) \right]
_{i,j=1}^{n}\right) ^{\ast } & \left[ \beta _{ij}1_{\mathcal{A}_{\lambda }}%
\right] _{i,j=1}^{n}%
\end{array}%
\right] =0
\end{equation*}%
in $M_{2}\left( M_{n}\left( \mathcal{A}_{\lambda }\right) \right) $ and
consequently, $\pi _{\lambda }^{\mathcal{A}}\left( a_{ij}\right) =\pi
_{\lambda }^{\mathcal{A}}\left( b_{ij}\right) _{i,j=1}^{n}=0$ and $\alpha
_{ij}=$ $\beta _{ij}=0$\ for all $i,j\in \{1,2,...,n\}.$ Since $\left\Vert
\left. \varphi ^{\left( n\right) }\left( A\right) \right\vert _{\mathcal{H}%
_{\iota }^{n}}\right\Vert \leq p_{\lambda }^{n}\left( A\right) $ for all $%
A\in M_{n}\left( \mathcal{M}\right) ,$ it follows that $\left. \varphi
^{\left( n\right) }\left( \left[ a_{ij}\right] _{i,j=1}^{n}\right)
\right\vert _{\mathcal{H}_{\iota }^{n}}=\left. \varphi ^{\left( n\right)
}\left( \left[ b_{ij}\right] _{i,j=1}^{n}\right) \right\vert _{\mathcal{H}%
_{\iota }^{n}}=0.$ Therefore, 
\begin{equation*}
\left. \Phi ^{\left( n\right) }\left( X\right) \right\vert _{\mathcal{H}%
_{\iota }^{n}\oplus \mathcal{H}_{\iota }^{n}}=\left[ 
\begin{array}{cc}
\left[ \alpha _{ij}\text{id}_{\mathcal{H}_{\iota }}\right] _{i,j=1}^{n} & 
\left. \varphi ^{\left( n\right) }\left( \left[ a_{ij}\right]
_{i,j=1}^{n}\right) \right\vert _{\mathcal{H}_{\iota }^{n}} \\ 
\left. \left( \varphi ^{\left( n\right) }\left( \left[ a_{ij}\right]
_{i,j=1}^{n}\right) \right) ^{\ast }\right\vert _{\mathcal{H}_{\iota }^{n}}
& \left[ \beta _{ij}\text{id}_{\mathcal{H}_{\iota }}\right] _{i,j=1}^{n}%
\end{array}%
\right] =\left[ 
\begin{array}{cc}
0 & 0 \\ 
0 & 0%
\end{array}%
\right] .
\end{equation*}
\end{proof}

\begin{remark}
The local $\mathcal{CP}$-map $\Phi :S_{\mathcal{M}}\rightarrow C^{\ast
}\left( \mathcal{D}_{\mathcal{E\oplus E}}\right) $ constructed in the above
lemma is unital and by \cite[Corolarry 4.1]{D}, it is a local $\mathcal{CC}$%
-map$.$
\end{remark}

Using the above lemma and Arveson's extension theorem (Theorem \ref{A}) we
extend the Wittstock's extension theorem \cite[Theorem 8.2]{P} in the
context of unbounded local $\mathcal{CC}$ -maps.

\begin{theorem}
Let $\mathcal{A}$ be a unital locally $C^{\ast }$-algebra, $\mathcal{%
M\subseteq }$ $\mathcal{A}$ be a local operator space, $\{\mathcal{H},%
\mathcal{E=\{H}_{n}\mathcal{\}}_{n},\mathcal{D}_{\mathcal{E}}\}\ $be a
quantized Fr\'{e}chet domain in a Hilbert space $\mathcal{H}$ and $\varphi :%
\mathcal{M\rightarrow }C^{\ast }\left( \mathcal{D}_{\mathcal{E}}\right) \ $%
be a local $\mathcal{CC}$-map. Then there exists a local $\mathcal{CC}$-map $%
\widetilde{\varphi }:$ $\mathcal{A\rightarrow }C^{\ast }\left( \mathcal{D}_{%
\mathcal{E}}\right) $ extending $\varphi .$
\end{theorem}

\begin{proof}
By Lemma \ref{1}, $\varphi :\mathcal{M\rightarrow }C^{\ast }\left( \mathcal{D%
}_{\mathcal{E}}\right) $ extends to a unital local $\mathcal{CP}$-map $\Phi
:S_{\mathcal{M}}\rightarrow C^{\ast }\left( \mathcal{D}_{\mathcal{E\oplus E}%
}\right) $ and by Theorem \ref{A}, $\Phi $ extends to a unital local $%
\mathcal{CP}$-map $\widetilde{\Phi }:M_{2}\left( \mathcal{A}\right)
\rightarrow C^{\ast }\left( \mathcal{D}_{\mathcal{E\oplus E}}\right) $.
Moreover, for each $a\in $ $\mathcal{M},$ $\varphi \left( a\right) $ can be
identified with $\widetilde{\Phi }\left( \left[ 
\begin{array}{cc}
0 & a \\ 
0 & 0%
\end{array}%
\right] \right) .$

For each $a\in \mathcal{A}$, there exists a unique element $\widetilde{%
\varphi }\left( a\right) $ such that $\widetilde{\Phi }\left( \left[ 
\begin{array}{cc}
0 & a \\ 
0 & 0%
\end{array}%
\right] \right) =\left[ 
\begin{array}{cc}
\ast & \widetilde{\varphi }\left( a\right) \\ 
\ast & \ast%
\end{array}%
\right] $. In this way, we obtained a linear map $\widetilde{\varphi }:%
\mathcal{A\rightarrow }C^{\ast }\left( \mathcal{D}_{\mathcal{E}}\right) .$
Moreover, $\left. \widetilde{\varphi }\right\vert _{\mathcal{M}}=\varphi .$
So $\widetilde{\varphi }$ extends $\varphi $. It remains to show that $%
\widetilde{\varphi }$ is a local $\mathcal{CC}$-map.

Let $\left[ a_{ij}\right] _{i,j=1}^{k}$. We have 
\begin{eqnarray*}
\left[ 
\begin{array}{cc}
\ast & \widetilde{\varphi }^{\left( k\right) }\left( \left[ a_{ij}\right]
_{i,j=1}^{k}\right) \\ 
\ast & \ast%
\end{array}%
\right] &=&\left[ \left[ 
\begin{array}{cc}
\ast & \widetilde{\varphi }\left( a_{ij}\right) \\ 
\ast & \ast%
\end{array}%
\right] \right] _{i,j=1}^{k} \\
&=&\left[ \widetilde{\Phi }\left( \left[ 
\begin{array}{cc}
0 & a_{ij} \\ 
0 & 0%
\end{array}%
\right] \right) \right] _{i,j=1}^{k}=\left[ \widetilde{\Phi }\left(
A_{ij}\right) \right] _{i,j=1}^{k} \\
&=&\widetilde{\Phi }^{\left( k\right) }\left( \left[ A_{ij}\right]
_{i,j=1}^{k}\right) .
\end{eqnarray*}%
Since $\widetilde{\Phi }$ is a unital local $\mathcal{CP}$-map, it is a
local $\mathcal{CC}$ -map \cite[Corollary 4.1]{D}. Then, for each $n$, there
exists $\lambda _{n}\in \Lambda $ such that 
\begin{eqnarray*}
\left\Vert \left. \widetilde{\varphi }^{\left( k\right) }\left( \left[ a_{ij}%
\right] _{i,j=1}^{k}\right) \right\vert _{\mathcal{H}_{n}^{k}}\right\Vert
&\leq &\left\Vert \left[ \left. \widetilde{\Phi }\left( A_{ij}\right)
\right\vert _{\mathcal{H}_{n}^{2}}\right] _{i,j=1}^{k}\right\Vert \leq
p_{\lambda _{n}}^{2k}\left( \left[ A_{ij}\right] _{i,j=1}^{k}\right) \\
&=&p_{\lambda _{n}}^{k}\left( \left[ a_{ij}\right] _{i,j=1}^{k}\right)
\end{eqnarray*}%
for all $\left[ a_{ij}\right] _{i,j=1}^{k}\in M_{k}\left( \mathcal{A}\right) 
$ and for all $k$. Therefore, $\widetilde{\varphi }$ is a local $\mathcal{CC}
$-map.
\end{proof}

The above theorem is a particular case of \cite[Theorem 8.1]{D}.

\section{Stinespring type theorem for local $\mathcal{CC}$-maps}

In \cite{D}, Dosiev proved a local convex version of well-known
Stinespring's dilation theorem for $\mathcal{CP}$-maps on $C^{\ast }$%
-algebras.

\begin{theorem}
\label{s} \cite[Theorem 5.1]{D} Let $\mathcal{A\ }$be a unital locally $%
C^{\ast }$-algebras and $\varphi :\mathcal{A}\rightarrow C^{\ast }(\mathcal{D%
}_{\mathcal{E}}))$ be a local $\mathcal{CCP}$-map. Then there exist a
quantized domain $\{\mathcal{H}^{\varphi },\mathcal{E}^{\varphi },\mathcal{D}%
_{\mathcal{E}^{\varphi }}\}$, where $\mathcal{E}^{\varphi }=\{\mathcal{H}%
_{\iota }^{\varphi };\iota \in \Upsilon \}$ is an upward filtered family of
closed subspaces of $\mathcal{H}^{\varphi }$, a contraction $V_{\varphi }:%
\mathcal{H}\rightarrow \mathcal{H}^{\varphi }$ and a unital local
contractive $\ast $-homomorphism $\pi _{\varphi }:\mathcal{A\rightarrow }%
C^{\ast }(\mathcal{D}_{\mathcal{E}^{\varphi }})$ such that

\begin{enumerate}
\item $V_{\varphi }\left( \mathcal{E}\right) \subseteq \mathcal{E}^{\varphi
};$

\item $\varphi \left( a\right) \subseteq V_{\varphi }^{\ast }\pi _{\varphi
}\left( a\right) V_{\varphi };$

\item $\mathcal{H}_{\iota }^{\varphi }=\left[ \pi _{\varphi }\left( \mathcal{%
A}\right) V_{\varphi }\mathcal{H}_{\iota }\right] $ for all $\iota \in
\Upsilon .$

Moreover, if $\varphi \left( 1_{\mathcal{A}}\right) =$id$_{\mathcal{D}_{%
\mathcal{E}}}$, then $V_{\varphi }$ is an isometry.
\end{enumerate}
\end{theorem}

The triple $\left( \pi _{\varphi },V_{\varphi },\{\mathcal{H}^{\varphi },%
\mathcal{E}^{\varphi },\mathcal{D}_{\mathcal{E}^{\varphi }}\}\right) $ is
called a minimal Stinespring dilation of $\varphi $. Moreover, the minimal
Stinespring dilation of $\varphi $ is unique up to unitary equivalence in
the following sense, if $\left( \pi _{\varphi },V_{\varphi },\{\mathcal{H}%
^{\varphi },\mathcal{E}^{\varphi },\mathcal{D}_{\mathcal{E}^{\varphi
}}\}\right) $ and$\left( \widetilde{\pi }_{\varphi },\widetilde{V}_{\varphi
},\{\widetilde{\mathcal{H}}^{\varphi },\widetilde{\mathcal{E}}^{\varphi },%
\widetilde{\mathcal{D}}_{\widetilde{\mathcal{E}}^{\varphi }}\}\right) $\ are
two minimal Stinespring dilations of $\varphi $, then there is a unitary
operator $U_{\varphi }:\mathcal{H}^{\varphi }\rightarrow \widetilde{\mathcal{%
H}}^{\varphi }$ such that $U_{\varphi }V_{\varphi }=\widetilde{V}_{\varphi }$
and $U_{\varphi }\pi _{\varphi }\left( a\right) \subseteq \widetilde{\pi }%
_{\varphi }\left( a\right) U_{\varphi }$ for all $a\in \mathcal{A\ }$\cite[%
Theorem 3.4]{BGK}.

If $\varphi :\mathcal{A}\rightarrow C^{\ast }(\mathcal{D}_{\mathcal{E}}))$\
is a local $\mathcal{CCP}$-map, then $\varphi \left( b(\mathcal{A})\right)
\subseteq b(C^{\ast }(\mathcal{D}_{\mathcal{E}}))$. Moreover, there is a $%
\mathcal{CCP}$- map $\left. \varphi \right\vert _{b(\mathcal{A})}:b(\mathcal{%
A})\rightarrow B(\mathcal{H})$ such that $\left. \left. \varphi \right\vert
_{b(\mathcal{A})}\left( a\right) \right\vert _{\mathcal{D}_{\mathcal{E}%
}}=\varphi \left( a\right) \ $for all $a\in b(\mathcal{A})\ $\cite[Remark 3.6%
]{MJ3}, and if $\left( \pi _{\varphi },V_{\varphi },\{\mathcal{H}^{\varphi },%
\mathcal{E}^{\varphi },\mathcal{D}_{\mathcal{E}^{\varphi }}\}\right) $ is a
minimal Stinespring dilation of $\varphi $, then $\left( \left. \pi
_{\varphi }\right\vert _{b(\mathcal{A})},V_{\varphi },\mathcal{H}^{\varphi
}\right) $, where $\left. \left. \pi _{\varphi }\right\vert _{b(\mathcal{A}%
)}\left( a\right) \right\vert _{\mathcal{D}_{\widetilde{\mathcal{E}}%
^{\varphi }}}=\pi _{\varphi }\left( a\right) \ $for all $a\in b(\mathcal{A}%
), $ is a minimal Stinespring dilation of $\left. \varphi \right\vert _{b(%
\mathcal{A})}$ \cite[Remark 3.8]{MJ3}.

In this section we will prove a local convex version of the Stinespring type
theorem for local $\mathcal{CC}$-maps. First we will show that, as in the
case of completely bounded maps on $C^{\ast }$-algebras, an unbounded local $%
\mathcal{CC}$-map on a unital locally $C^{\ast }$-algebra $\mathcal{A}$
induces a unital local $\mathcal{CP}$- map on $M_{2}\left( \mathcal{A}%
\right) .$

\begin{theorem}
\label{2} Let $\mathcal{A}$ be a unital locally $C^{\ast }$-algebra, $\{%
\mathcal{H},\mathcal{E=\{H}_{n}\mathcal{\}}_{n},\mathcal{D}_{\mathcal{E}}\}\ 
$be a quantized Fr\'{e}chet domain in a Hilbert space $\mathcal{H}$ and $%
\varphi :\mathcal{A\rightarrow }C^{\ast }\left( \mathcal{D}_{\mathcal{E}%
}\right) \ $be a local $\mathcal{CC}$-map. Then there exist two unital local 
$\mathcal{CP}$-maps $\varphi _{i}:A\rightarrow C^{\ast }\left( \mathcal{D}_{%
\mathcal{E}}\right) ,i=1,2,$ such that $\Phi :M_{2}\left( \mathcal{A}\right)
\rightarrow C^{\ast }\left( \mathcal{D}_{\mathcal{E\oplus E}}\right) \ $%
given by 
\begin{equation*}
\Phi \left( \left[ 
\begin{array}{cc}
a & b \\ 
c & d%
\end{array}%
\right] \right) =\left[ 
\begin{array}{cc}
\varphi _{1}\left( a\right) & \varphi \left( b\right) \\ 
\varphi \left( c\right) ^{\ast } & \varphi _{2}\left( d\right)%
\end{array}%
\right]
\end{equation*}%
is a unital local $\mathcal{CP}$-map.
\end{theorem}

\begin{proof}
By Lemma \ref{1}, there exists a unital local $\mathcal{CP}$-map $\Phi _{1}:%
\mathcal{S}_{\mathcal{A}}\rightarrow C^{\ast }\left( \mathcal{D}_{\mathcal{%
E\oplus E}}\right) $ such that%
\begin{equation*}
\Phi _{1}\left( \left[ 
\begin{array}{cc}
\alpha 1_{\mathcal{A}} & a \\ 
b & \beta 1_{\mathcal{A}}%
\end{array}%
\right] \right) =\left[ 
\begin{array}{cc}
\alpha 1_{C^{\ast }\left( \mathcal{D}_{\mathcal{E}}\right) } & \varphi
\left( a\right) \\ 
\varphi \left( b\right) ^{\ast } & \beta 1_{C^{\ast }\left( \mathcal{D}_{%
\mathcal{E}}\right) }%
\end{array}%
\right]
\end{equation*}%
and by Theorem \ref{A}, there exists a unital local $\mathcal{CP}$-map $\Phi
:M_{2}\left( \mathcal{A}\right) \rightarrow C^{\ast }\left( \mathcal{D}_{%
\mathcal{E\oplus E}}\right) $ such that $\left. \Phi \right\vert _{\mathcal{S%
}_{\mathcal{A}}}=\Phi _{1}.$

Since $\Phi $ is a unital local $\mathcal{CP}$-map, by \cite[Corrolary 4.1]%
{D}, $\Phi $ is a local $\mathcal{CC}$-map. Therefore, $\Phi $ is a
continuous completely positive map. Let $a\in \mathcal{A}$ such that $0\leq
a\leq 1_{\mathcal{A}}.$ Then 
\begin{eqnarray*}
\left[ 
\begin{array}{cc}
0 & 0 \\ 
0 & 0%
\end{array}%
\right] &\leq &\Phi \left( \left[ 
\begin{array}{cc}
a & 0 \\ 
0 & 0%
\end{array}%
\right] \right) \leq \Phi \left( \left[ 
\begin{array}{cc}
1_{\mathcal{A}} & 0 \\ 
0 & 0%
\end{array}%
\right] \right) \\
&=&\Phi _{1}\left( \left[ 
\begin{array}{cc}
1_{\mathcal{A}} & 0 \\ 
0 & 0%
\end{array}%
\right] \right) =\left[ 
\begin{array}{cc}
1_{C^{\ast }\left( \mathcal{D}_{\mathcal{E}}\right) } & 0 \\ 
0 & 0%
\end{array}%
\right] .
\end{eqnarray*}%
Therefore, 
\begin{equation*}
\Phi \left( \left[ 
\begin{array}{cc}
a & 0 \\ 
0 & 0%
\end{array}%
\right] \right) =\left[ 
\begin{array}{cc}
\ast & 0 \\ 
0 & 0%
\end{array}%
\right] .
\end{equation*}%
Since $b(\mathcal{A})_{+}$ is a closed cone in $b(\mathcal{A}),$ $b(\mathcal{%
A})$ is dense in $\mathcal{A}$ and $\Phi $ is continuous, 
\begin{equation*}
\Phi \left( \left[ 
\begin{array}{cc}
a & 0 \\ 
0 & 0%
\end{array}%
\right] \right) =\left[ 
\begin{array}{cc}
\ast & 0 \\ 
0 & 0%
\end{array}%
\right]
\end{equation*}%
for all $a\in \mathcal{A}$. Therefore, for each $a\in \mathcal{A}$, there
exists a unique element $\varphi _{1}\left( a\right) $ such that 
\begin{equation*}
\Phi \left( \left[ 
\begin{array}{cc}
a & 0 \\ 
0 & 0%
\end{array}%
\right] \right) =\left[ 
\begin{array}{cc}
\varphi _{1}\left( a\right) & 0 \\ 
0 & 0%
\end{array}%
\right] .
\end{equation*}%
In this way, we obtained a linear map $\varphi _{1}:\mathcal{A\rightarrow }%
C^{\ast }\left( \mathcal{D}_{\mathcal{E}}\right) $.$\ $Clearly, $\varphi
_{1}\left( 1_{\mathcal{A}}\right) =1_{C^{\ast }\left( \mathcal{D}_{\mathcal{E%
}}\right) }$.$\ $To show that $\varphi _{1}$ is a local $\mathcal{CP}$-map,
since it is unital, it is sufficient to show that $\varphi _{1}$ is a local $%
\mathcal{CC}$-map.$\ $Let $n\in \mathbb{N}$. Since $\Phi $ is a local $%
\mathcal{CC}$-map, there exists $\lambda _{n}\in \Lambda $ such that 
\begin{equation*}
\left\Vert \left. \Phi ^{\left( k\right) }\left( \left[ A_{ij}\right]
_{i,j=1}^{k}\right) \right\vert _{\mathcal{H}_{n}^{2k}}\right\Vert \leq
p_{\lambda _{n}}^{2k}\left( \left[ A_{ij}\right] _{i,j=1}^{k}\right)
\end{equation*}%
for all $\left[ A_{ij}\right] _{i,j=1}^{k}\in M_{k}\left( M_{2}\left( 
\mathcal{A}\right) \right) $ and for all $k$. Then 
\begin{eqnarray*}
\left\Vert \left. \varphi _{1}^{\left( k\right) }\left( \left[ a_{ij}\right]
_{i,j=1}^{n}\right) \right\vert _{\mathcal{H}_{n}^{k}}\right\Vert &\leq
&\left\Vert \left. \Phi ^{\left( k\right) }\left( \left[ A_{ij}\right]
_{i,j=1}^{k}\right) \right\vert _{\mathcal{H}_{n}^{k}\oplus \mathcal{H}%
_{n}^{k}}\right\Vert \\
(A_{ij} &=&\left[ 
\begin{array}{cc}
a_{ij} & 0 \\ 
0 & 0%
\end{array}%
\right] ) \\
&\leq &p_{\lambda _{n}}^{2k}\left( \left[ A_{ij}\right] _{i,j=1}^{k}\right)
=p_{\lambda _{n}}^{k}\left( \left[ a_{ij}\right] _{i,j=1}^{k}\right)
\end{eqnarray*}%
for all $\left[ a_{ij}\right] _{i,j=1}^{k}\in M_{k}\left( \mathcal{A}\right) 
$ and for all $k$, and so $\varphi _{1}$ is a local $\mathcal{CC}$-map. In
the same way, we show that there exists a unital local $\mathcal{CP}$-map $%
\varphi _{2}:\mathcal{A\rightarrow }C^{\ast }\left( \mathcal{D}_{\mathcal{E}%
}\right) $ such that 
\begin{equation*}
\Phi \left( \left[ 
\begin{array}{cc}
0 & 0 \\ 
0 & b%
\end{array}%
\right] \right) =\left[ 
\begin{array}{cc}
0 & 0 \\ 
0 & \varphi _{2}\left( b\right)%
\end{array}%
\right] .
\end{equation*}%
We have 
\begin{eqnarray*}
\Phi \left( \left[ 
\begin{array}{cc}
a & b \\ 
c & d%
\end{array}%
\right] \right) &=&\Phi \left( \left[ 
\begin{array}{cc}
a & 0 \\ 
0 & 0%
\end{array}%
\right] \right) +\Phi \left( \left[ 
\begin{array}{cc}
0 & 0 \\ 
0 & d%
\end{array}%
\right] \right) +\Phi \left( \left[ 
\begin{array}{cc}
0 & b \\ 
c & 0%
\end{array}%
\right] \right) \\
&=&\left[ 
\begin{array}{cc}
\varphi _{1}\left( a\right) & 0 \\ 
0 & 0%
\end{array}%
\right] +\left[ 
\begin{array}{cc}
0 & 0 \\ 
0 & \varphi _{2}\left( d\right)%
\end{array}%
\right] +\Phi _{1}\left( \left[ 
\begin{array}{cc}
0 & b \\ 
c & 0%
\end{array}%
\right] \right) \\
&=&\left[ 
\begin{array}{cc}
\varphi _{1}\left( a\right) & 0 \\ 
0 & \varphi _{2}\left( d\right)%
\end{array}%
\right] +\left[ 
\begin{array}{cc}
0 & \varphi \left( b\right) \\ 
\varphi \left( c\right) ^{\ast } & 0%
\end{array}%
\right] \\
&=&\left[ 
\begin{array}{cc}
\varphi _{1}\left( a\right) & \varphi \left( b\right) \\ 
\varphi \left( c\right) ^{\ast } & \varphi _{2}\left( d\right)%
\end{array}%
\right]
\end{eqnarray*}%
for all $a,b,c,d\in \mathcal{A}$, and the theorem is proved.
\end{proof}

The following theorem is a local convex version of the Stinesping type
theorem for completely bounded maps on $C^{\ast }$-algebras.

\begin{theorem}
\label{3} Let $\mathcal{A}$ be a unital locally $C^{\ast }$-algebra, $%
\mathcal{E=\{H}_{n}\mathcal{\}}_{n}\ $be a quantized Fr\'{e}chet domain in a
Hilbert space $\mathcal{H}$ and $\varphi :\mathcal{A\rightarrow }C^{\ast
}\left( \mathcal{D}_{\mathcal{E}}\right) \ $be a local $\mathcal{CC}$-map.
Then there exist a quantized Frech\'{e}t domain $\mathcal{E}^{\varphi }%
\mathcal{=\{H}_{n}^{\varphi }\mathcal{\}}_{n}\ $in a Hilbert space $\mathcal{%
H}^{\varphi }$, a contractive unital $\ast $-morphism $\pi _{\varphi }:%
\mathcal{A\rightarrow }C^{\ast }\left( \mathcal{D}_{\mathcal{E}^{\varphi
}}\right) $, and two isometries $V_{i}:\mathcal{H}\rightarrow \mathcal{H}%
^{\varphi },$ $V_{i}\left( \mathcal{E}\right) \subseteq \mathcal{E}^{\varphi
},i=1,2$, such that 
\begin{equation*}
\varphi \left( a\right) =V_{1}^{\ast }\pi _{\varphi }\left( a\right) V_{2}
\end{equation*}%
for all $a\in \mathcal{A}$.
\end{theorem}

\begin{proof}
By Theorem \ref{2}, there exist two unital local $\mathcal{CP}$- maps $%
\varphi _{i}:A\rightarrow C^{\ast }\left( \mathcal{D}_{\mathcal{E}}\right)
,i=1,2,$ such that $\Phi :M_{2}\left( \mathcal{A}\right) \rightarrow C^{\ast
}\left( \mathcal{D}_{\mathcal{E\oplus E}}\right) \ $given by 
\begin{equation*}
\Phi \left( \left[ 
\begin{array}{cc}
a & b \\ 
c & d%
\end{array}%
\right] \right) =\left[ 
\begin{array}{cc}
\varphi _{1}\left( a\right) & \varphi \left( b\right) \\ 
\varphi \left( c\right) ^{\ast } & \varphi _{2}\left( d\right)%
\end{array}%
\right]
\end{equation*}%
is a unital local $\mathcal{CP}$-map. Let $\left( \Pi ,\{\mathcal{K}^{\Phi },%
\mathcal{E}^{\Phi },\mathcal{D}_{\mathcal{E}^{\Phi }}\},V^{\Phi }\right) ,$
where $\mathcal{E}^{\Phi }=$ $\{\mathcal{K}_{n}^{\Phi }\}_{n},$ be a minimal
Stinespring dilation of $\Phi .$ Since $\Phi $ is unital, $V$ is an
isometry. Then $\left( \left. \Pi \right\vert _{M_{2}\left( b(\mathcal{A}%
)\right) },\mathcal{K}^{\Phi },V^{\Phi }\right) $ is a Stinespring dilation
of $\left. \Phi \right\vert _{M_{2}(b(\mathcal{A}))}=\left[ 
\begin{array}{cc}
\left. \varphi _{1}\right\vert _{b(\mathcal{A})} & \left. \varphi
\right\vert _{b(\mathcal{A})} \\ 
\left. \varphi ^{\ast }\right\vert _{b(\mathcal{A})} & \left. \varphi
_{2}\right\vert _{b(\mathcal{A})}%
\end{array}%
\right] $, where $\varphi ^{\ast }\left( a\right) =\varphi \left( a\right)
^{\ast }$, and by \cite[Theorem 8.4]{P}, there exist a Hilbert space $%
\mathcal{H}^{\varphi }$, two isometries $V_{1},V_{2}\in B(\mathcal{H},%
\mathcal{H}^{\varphi })$ and a unital $\ast $-morphism $\pi :b(\mathcal{A}%
)\rightarrow B(\mathcal{H}^{\varphi })$ such that 
\begin{eqnarray*}
\Phi \left( \left[ 
\begin{array}{cc}
a & b \\ 
c & d%
\end{array}%
\right] \right) &=&\left[ 
\begin{array}{cc}
\varphi _{1}\left( a\right) & \varphi \left( b\right) \\ 
\varphi \left( c\right) ^{\ast } & \varphi _{2}\left( d\right)%
\end{array}%
\right] \\
&=&\left[ 
\begin{array}{cc}
V_{1}^{\ast } & 0 \\ 
0 & V_{2}^{\ast }%
\end{array}%
\right] \left[ 
\begin{array}{cc}
\pi \left( a\right) & \pi \left( b\right) \\ 
\pi \left( c\right) ^{\ast } & \pi \left( d\right)%
\end{array}%
\right] \left[ 
\begin{array}{cc}
V_{1} & 0 \\ 
0 & V_{2}%
\end{array}%
\right] \\
&=&\left[ 
\begin{array}{cc}
V_{1}^{\ast }\pi \left( a\right) V_{1} & V_{1}^{\ast }\pi \left( b\right)
V_{2} \\ 
V_{2}^{\ast }\pi \left( c\right) ^{\ast }V_{1} & V_{2}^{\ast }\pi \left(
d\right) V_{2}%
\end{array}%
\right]
\end{eqnarray*}%
for all $a,b,c,d\in b(\mathcal{A)}$. Moreover\textbf{,}$\ \mathcal{H}%
^{\varphi }=\overline{\text{span}}\{\pi \left( a\right) V_{1}\xi +\pi \left(
b\right) V_{2}\eta ;a,b\in b(\mathcal{A)},\xi ,\eta \in \mathcal{H}\}$.
Therefore, $\varphi \left( a\right) =V_{1}^{\ast }\pi \left( a\right) V_{2}\ 
$for all $a\in b(\mathcal{A)}.$

For each $n,\ $let $\mathcal{H}_{n}^{\varphi }=\overline{\text{span}}\{\pi
\left( a\right) V_{1}\xi +\pi \left( b\right) V_{2}\eta ;a,b\in b(\mathcal{A)%
},\xi ,\eta \in \mathcal{H}_{n}\}$. Clearly, $\mathcal{H}_{n}^{\varphi }$ is
a Hilbert subspace of $\mathcal{H}^{\varphi }$ and $\mathcal{E}^{\varphi }=\{%
\mathcal{H}_{n}^{\varphi }\}_{n}\ $is a quantized domain in $\mathcal{H}%
^{\varphi }.\ $Moreover, $V_{i}\left( \mathcal{H}_{n}\right) \subseteq 
\mathcal{H}_{n}^{\varphi },i=1,2$ and $\pi \left( a\right) \left( \mathcal{H}%
_{n}^{\varphi }\right) \subseteq \mathcal{H}_{n}^{\varphi }$ for all $a\in b(%
\mathcal{A)}$ and for all $n$. Therefore, $V_{i}\left( \mathcal{E}\right)
\subseteq \mathcal{E}^{\varphi },i=1,2$, and $\pi \left( a\right) \in
C^{\ast }\left( \mathcal{D}_{\mathcal{E}^{\varphi }}\right) \cap B(\mathcal{H%
}^{\varphi })$ for all $a\in b(\mathcal{A)}$. Since $\Pi $ is a unital local 
$\ast $ -morphism, for each $n$, there exists $\lambda _{n}\in \Lambda $
such that 
\begin{equation*}
\left\Vert \left. \pi \left( a\right) \right\vert _{\mathcal{H}_{n}^{\varphi
}}\right\Vert \leq \left\Vert \left. \Pi \left( \left[ 
\begin{array}{cc}
a & 0 \\ 
0 & 0%
\end{array}%
\right] \right) \right\vert _{\mathcal{H}_{n}^{\varphi }\oplus \mathcal{H}%
_{n}^{\varphi }}\right\Vert \leq p_{\lambda _{n}}^{2}\left( \left[ 
\begin{array}{cc}
a & 0 \\ 
0 & 0%
\end{array}%
\right] \right) =p_{\lambda _{n}}\left( a\right)
\end{equation*}%
for all $a\in b(\mathcal{A)}$. From the above relation, we deduce that the
unital $\ast $-morphism $\pi :b(\mathcal{A})\rightarrow B(\mathcal{H}%
^{\varphi })$ is continuous with respect to the families of $C^{\ast }$%
-seminorms $\{\left. p_{\lambda }\right\vert _{b(\mathcal{A})}\}_{\lambda
\in \Lambda }$ and $\{\left\Vert \cdot \right\Vert _{B(\mathcal{H}%
_{n}^{\varphi })}\}_{n}$. Therefore, $\pi $ extends to a unital local $\ast $
-morphism from $\mathcal{A}$ to $C^{\ast }\left( \mathcal{D}_{\mathcal{E}%
^{\varphi }}\right) $, denoted by $\pi _{\varphi }$.

From $\varphi \left( a\right) =V_{1}^{\ast }\pi \left( a\right) V_{2}\
=V_{1}^{\ast }\pi _{\varphi }\left( a\right) V_{2}\ $for all $a\in b(%
\mathcal{A)}$ and taking into account that $\varphi $ and $\pi _{\varphi }$
are local contractive, we deduce that 
\begin{equation*}
\varphi \left( a\right) =V_{1}^{\ast }\pi _{\varphi }\left( a\right) V_{2}\ 
\end{equation*}%
for all $a\in \mathcal{A}.$
\end{proof}

\begin{remark}
It is known the Wittstock's decomposition theorem for bounded operator
valued completely bounded maps on $C^{\ast }$ that states that any
completely bounded map $\varphi $ from a $C^{\ast }$-algebra $\mathcal{A}$
to $B(\mathcal{H})$ is the sum of four completely positive maps. Using
Theorem \ref{3}, and following the proof of Wittstock's decomposition
theorem \cite[Theorem 8.5]{P}, we obtain a similar result for unbounded
local $\mathcal{CC}$-maps: any local $\mathcal{CC}$-map $\varphi $ from a
unital locally $C^{\ast }$-algebra $\mathcal{A}$ to $C^{\ast }\left( 
\mathcal{D}_{\mathcal{E}}\right) $, where $\mathcal{E=\{H}_{n}\mathcal{\}}%
_{n}\ $is a quantized Fr\'{e}chet domain in a Hilbert space $\mathcal{H}$,
is the sum of four local $\mathcal{CC}$-maps.
\end{remark}

\end{document}